\documentclass{amsart}
\usepackage[english]{babel}
\usepackage[latin1]{inputenc}
\usepackage[dvips,final]{graphics}
\usepackage{amsmath,amsfonts,amssymb,amsthm,amscd,array,stmaryrd,mathrsfs, mathdots, epigraph}
\usepackage{arydshln}
\usepackage[makeroom]{cancel}
\usepackage{pstricks}
 \usepackage[all]{xy}
 \usepackage{url}
\usepackage{multirow, blkarray}
\usepackage{booktabs}
\usepackage{textcomp}
 \usepackage[final]{epsfig}
 \usepackage{color}
\theoremstyle{plain}
\newtheorem{thm}{Theorem}
\newtheorem*{thmc}{Theorem}
\newtheorem{lem}{Lemma}

\newtheorem{prop}[lem]{Proposition}

\theoremstyle{definition}

\newtheorem{ex}[lem]{Example}

\newcommand{\R}{\mathbb{R}}
\newcommand{\Z}{\mathbb{Z}}

\newcommand{\tr}{\textup{tr }}

\newcommand{\pf}{\mathrm{pf}}
\newcommand{\Id}{\mathrm{Id}}
\newcommand{\SL}{\mathrm{SL}}

\newcommand{\midmatrix}{\textup{mid}}
\newcommand{\thup}{\textup{th}}

\begin{document}

\title{Rotundus: triangulations, Chebyshev polynomials, and Pfaffians}

\author{Charles H.\ Conley}
\address{
Charles H.\ Conley,
Department of Mathematics 
\\University of North Texas 
\\Denton TX 76203, USA} 
\email{conley@unt.edu}

\author{Valentin Ovsienko}
\address{
Valentin Ovsienko,
CNRS,
Laboratoire de Math\'ematiques 
U.F.R. Sciences Exactes et Naturelles 
Moulin de la Housse - BP 1039 
51687 REIMS cedex 2,
France}
\email{valentin.ovsienko@univ-reims.fr}



\begin{abstract}
We introduce and study a cyclically invariant polynomial which is an analog of the
classical tridiagonal determinant usually called the continuant.
We prove that this polynomial can be calculated as the Pfaffian of a skew-symmetric matrix.
We consider the corresponding Diophantine equation and prove an analog of a famous
result due to Conway and Coxeter.
We also observe that Chebyshev polynomials of the first kind
arise as Pfaffians.
\end{abstract}

\maketitle

\thispagestyle{empty}

The tridiagonal determinant 
\begin{equation}
\label{ContEq}
K_{n}(a_1,\ldots,a_n):=
\det\left(
\begin{array}{cccccc}
a_1&1&&&\\[4pt]
1&a_{2}&1&&\\[4pt]
&\ddots&\ddots&\!\!\ddots&\\[4pt]
&&1&a_{n-1}&\!\!\!\!\!1\\[4pt]
&&&\!\!\!\!\!1&\!\!\!\!a_{n}
\end{array}
\right)
\end{equation}
is most often known as the {\it continuant.\/}
It has a long and enchanting history.
Let us mention a few of its many interesting properties.

\begin{enumerate}
\item[a)]
The continuant was already known to Euler,
although the notion of determinant was not in use in his time;
see~\cite{Eul}, Chapter~18.
Indeed, continuants occur
as both the numerator and the denominator of continued fractions:
$$
a_1 - \cfrac{1}{a_2 
          - \cfrac{1}{\ddots - \cfrac{1}{a_n} } } 
          \quad=\quad
\frac{K_{n}(a_1,\ldots,a_n)}{K_{n-1}(a_2,\ldots,a_n)}.
$$
In the course of studying this formula
Euler discovered a simple algorithm for calculating continuants,
which we recall in Section~\ref{ContSec}.
He went on to prove a series of identities involving them.

\smallbreak \item[b)]
The matrix formula
\begin{equation}
\label{SLBisEq}
M_n:=\left(
\begin{array}{cc}
K_{n}(a_1,\ldots,a_n)&K_{n-1}(a_1,\ldots,a_{n-1})\\[4pt]
-K_{n-1}(a_2,\ldots,a_n)&-K_{n-2}(a_2,\ldots,a_{n-1})
\end{array}
\right)
=
\left(
\begin{array}{cc}
a_1&1\\[4pt]
-1&0
\end{array}
\right)
\cdots
\left(
\begin{array}{cc}
a_n&1\\[4pt]
-1&0
\end{array}
\right)
\end{equation}
puts continuants in the context of $\SL(2,\R)$,
and even $\SL(2,\Z)$ when the $a_i$ are integral.

\smallbreak\item[c)]
Continuants are related to the spectral theory of
difference equations.  Indeed, they can be defined
in terms of solutions of the linear difference equation
\begin{equation}
\label{DEqEq}
V_{i-1}-a_iV_{i}+V_{i+1}=0,
\end{equation}
known as the discrete Sturm-Liouville, Hill, or Schr\"odinger equation:
the initial conditions $(V_{0},V_{1})=(0,1)$ give
$V_{n+1} = K_{n}(a_1,\ldots,a_n)$.
If the sequence $(a_i)_{i\in\Z}$ is $n$-periodic, then
the matrix $M_n$ in~(\ref{SLBisEq}) is the monodromy matrix of~(\ref{DEqEq}).

\smallbreak \item[d)]
Continuants appeared in the work of Coxeter~\cite{Cox}
as the values of frieze patterns (for a survey, see~\cite{Sop}).
For $(a_i)$ $n$-periodic,
Conway and Coxeter~\cite{CoCo} considered the Diophantine system
\begin{equation}
\label{CoCotEq}
K_{n-2}(a_i,\ldots,a_{i+n-3})=1, \quad i \in \Z.
\end{equation}
(Of course, due to the periodicity there are only $n$ distinct equations.)
It can be shown that this system is equivalent to the condition that the monodromy matrix
$M_n$ of~(\ref{DEqEq}) is~$-\Id$.
Conway and Coxeter proved the beautiful theorem
that every {\it totally positive\/} $n$-periodic integer solution
$(a_i)$ of this system corresponds to a triangulation of an
$n$-gon\footnote{Conway and Coxeter called such a solution a {\it quiddity}.}.
This implies in particular that such solutions are enumerated by the Catalan numbers.
For details, see Section~\ref{TotSec}.

\smallbreak \item[e)]
As discussed in~\cite{Aig}, continuants
have another property related to the Catalan numbers.
Given any sequence $a=(a_0,a_1,a_2,\ldots)$,
there exists a unique sequence $C=(C_0,C_1,C_2,\ldots)$
determined by the condition that the {\it Hankel matrices\/}
$$
A_n :=
\left(
\begin{array}{cccc}
C_0&C_1&\cdots&C_n\\
C_1&C_2&\cdots&C_{n+1}\\
\vdots&\vdots&&\vdots\\
C_n&C_{n+1}&\cdots&C_{2n}
\end{array}
\right),
\qquad
B_n :=
\left(
\begin{array}{cccc}
C_1&C_2&\cdots&C_n\\
C_2&C_3&\cdots&C_{n+1}\\
\vdots&\vdots&&\vdots\\
C_n&C_{n+1}&\cdots&C_{2n-1}
\end{array}
\right)
$$
have determinants $\det(A_n)=1$ and $\det(B_n)=K_{n+1}(a_0,\ldots,a_n)$.
The sequence $a = (1,2,2,2,\ldots)$ has $K_{n+1}(1, 2, 2, \ldots, 2) = 1$ for all $n > 0$
and determines the Catalan numbers.
\end{enumerate}

\medbreak
Among all the wonderful properties of the continuant, there is one which might be
considered a flaw: it is not invariant under cyclic permutations of its arguments.
Indeed, the polynomials
$$
K_{n}(a_1,\ldots,a_n),\quad
K_{n}(a_n,a_1,\ldots,a_{n-1}),\quad
\ldots,\quad
K_{n}(a_2,\ldots,a_n,a_1)
$$
are all different.
At times this can be inconvenient.
For instance, in considering the Conway-Coxeter system~(\ref{CoCotEq}),
one has to deal with $n$ equations.

In this note, we introduce a cyclically invariant version of continuants.

\medbreak \noindent {\sc Comment.}
The history of the term ``continuant'' in this setting is amusing.
The polynomial~$K_n$ was baptized thus by Muir,
who had discovered it independently,
only to learn later that Sylvester and others had discovered it earlier.
Muir's choice of name was severely contested by Sylvester,
who wrote in a letter to Clifford
{\it I protest against my most expressive and suggestive word ``cumulants''
being ignored by Mr.~Muir and replaced
by the unmeaning and ill chosen word ``continuants''.\/}
Muir responded in the letter~\cite{Mui},
written in the enjoyable style that has unfortunately since been lost
in mathematical communications, that the name was chosen
{\it
 (1) because, as an exceedingly suitable and euphonious abbreviation for 
 ``continued-fraction determinant'', it seems to me to be the very word wanted,
 (2) because, in this way, it is a short literal translation of the equivalent term 
 ``Kettenbruch-Determinante'', which is the received name in Germany, 
 (3) because, though it may be somewhat scant of meaning to a literalist, 
 I cannot but consider it eminently ``suggestive'', and
 (4) because doubtless I have still a foster-father's kindly
 feeling towards the name he has known another's child by.\/}
While Sylvester responded
{\it Reasons 2 and 3 above given appear to afford
quite a sufficient justification for the use of the word in question,\/}
we might add that Reason~4 cannot be underestimated!

\section{Introducing the Rotundus}\label{IntSec}
We set
\begin{equation}
\label{RotDefEq}
R_n (a_1,\ldots,a_n):=K_n (a_1,\ldots,a_n)-
K_{n-2} (a_2,\ldots,a_{n-1}).
\end{equation}
Note that this polynomial is nothing other than the trace of the matrix~(\ref{SLBisEq}).
The first examples are
$$
\begin{array}{rcl}
R_1(a) &=& a,\\[4pt]
R_2(a_1,a_2) &=& a_1a_2-2,\\[4pt]
R_3(a_1,a_2,a_3) &=& a_1a_2a_3-a_1-a_2-a_3,\\[4pt]
R_4(a_1,a_2,a_3,a_4) &=& a_1a_2a_3a_4-a_1a_2-a_2a_3-a_3a_4-a_1a_4+2,\\[4pt]
R_5(a_1,a_2,a_3,a_4,a_5) &=& 
a_1a_2a_3a_4a_5\\
&&-a_1a_2a_3-a_2a_3a_4-a_3a_4a_5-a_1a_4a_5-a_1a_2a_5\\
&&+a_1+a_2+a_3+a_4+a_5.
\\
\end{array}
$$

\begin{prop}
\label{CyProp}
$R_n$ is cyclically invariant:
$R_n (a_1,\ldots,a_n)=R_n (a_n,a_1,\ldots,a_{n-1})$.
\end{prop}

\begin{proof}
This is an immediate consequence of Euler's algorithm, given in Section~\ref{ContSec} below.
\end{proof}

In light of this proposition, we suggest the Latin term {\it rotundus\/} as a name for $R_n$.
We will show that several properties of the rotundus
are, in fact, more sophisticated versions of analogous properties of the continuant $K_n$.
For instance, in Section~\ref{GrammSec} we calculate $R_n$ as a Pfaffian.
Speaking ``philosophically'', the relation of~$R_n$ and~$K_n$
is similar to that of the Chebyshev polynomials of the first and second kinds:
see Section~\ref{ChebSec}.

\section{The cyclic Euler algorithm}\label{ContSec}

Euler's algorithm for calculating the continuant
$K_n(a_1,\ldots,a_n)$ is as follows:
start with the full product $a_1 \ldots a_{n}$
and successively replace all the adjacent pairs $a_ia_{i+1}$ by $-1$
in all possible ways.
For example,
$$
\begin{array}{rcl}
K_3(a_1,a_2,a_3)&=&a_1a_2a_3-\cancel{a_1a_2}a_3-a_1\cancel{a_2a_3}
=a_1a_2a_3-a_1-a_3,\\[4pt]
K_4(a_1,a_2,a_3,a_4)&=&
a_1a_2a_3a_4-
\cancel{a_1a_2}a_3a_4-a_1\cancel{a_2a_3}a_4-a_1a_2\cancel{a_3a_4}
+\cancel{a_1a_2}\cancel{a_3a_4}\\
&=&
a_1a_2a_3a_4-a_1a_2-a_1a_4-a_3a_4+1.
\end{array}
$$

It follows directly from~(\ref{RotDefEq}) that
the rotundus is calculated by nearly the same rule.
The only difference is that the variables are ordered cyclically,
so the pair $a_na_1$ is considered adjacent.
For example, 
$$
\begin{array}{rcl}
R_3(a_1,a_2,a_3)&=&a_1a_2a_3-\cancel{a_1a_2}a_3-a_1\cancel{a_2a_3}
-\cancel{a_1}a_2\cancel{a_3}\\
&=&a_1a_2a_3-a_1-a_2-a_3,\\[4pt]
R_4(a_1,a_2,a_3,a_4)&=&
a_1a_2a_3a_4-
\cancel{a_1a_2}a_3a_4-a_1\cancel{a_2a_3}a_4-a_1a_2\cancel{a_3a_4}
-\cancel{a_1}a_2a_3\cancel{a_4}\\
&&+\cancel{a_1a_2}\cancel{a_3a_4}+\cancel{a_1}\cancel{a_2a_3}\cancel{a_4}\\
&=&
a_1a_2a_3a_4-a_1a_2-a_1a_4-a_2a_3-a_3a_4+2.
\end{array}
$$
At order~$5$ one has
$$\begin{array}{rcl}
R_5(a_1,a_2,a_3,a_4,a_5)&=&
a_1a_2a_3a_4a_5-
\cancel{a_1a_2}a_3a_4a_5-\cdots
-\cancel{a_1}a_2a_3a_4\cancel{a_5}\\
&&+\cancel{a_1}a_2\cancel{a_3a_4}\cancel{a_5}
+\cdots+a_1\cancel{a_2a_3}\cancel{a_4a_5}\\[4pt]
&=&
a_1a_2a_3a_4a_5
-a_1a_2a_3-a_2a_3a_4-a_3a_4a_5-a_1a_4a_5-a_1a_2a_5\\
&&+a_1+a_2+a_3+a_4+a_5.
\end{array}
$$
Clearly the second term on the right side of~(\ref{RotDefEq})
contains precisely all those terms in the modified algorithm with $a_na_1$ removed.
We refer to this procedure as the ``cyclic Euler algorithm''.

\section{Pfaffians}\label{GrammSec}

Recall that the determinant of a skew-symmetric matrix $\Omega$ is
the square of a certain polynomial in its entries, known as the {\it Pfaffian}:
$$
\det(\Omega)=:\pf(\Omega)^2.
$$
It turns out that the rotundus
is the Pfaffian of a very simple skew-symmetric matrix of size $2n\times 2n$:

\begin{thm}
\label{PfLem}
One has
\begin{equation}
\label{TheOmEq}
\det\left(
\begin{array}{cccccccccc}
&&&\;\;1&a_1&1&\\[2pt]
&&&&1&a_2&1&\\
&&&&&\ddots&\ddots&\ddots\\
&&&&&&\ddots&\ddots&1\\
-1&&&&&&&1&a_{n}\\[2pt]
-a_1&-1&&&&&&&1\\
-1&\ddots&\!\!\!\!\!\!\ddots&&&&&&\\
&\ddots&\ddots&\\
&&&\;\;-1&&&&\\[2pt]
&&-1&-a_{n}&\!\!-1&&
\end{array}
\right) = R_n(a_1,\ldots,a_n)^2.
\end{equation}
\end{thm}

This formula may be understood as an analog of~(\ref{ContEq}).  It is entertaining to prove the cyclic symmetry of the determinant directly by conjugating by the appropriate permutation matrices.

\medbreak \noindent {\sc Example.}
One can easily check directly that
$$
\pf\left(
\begin{array}{cccccc}
0&0&1&a_1&1&0\\[4pt]
0&0&0&1&a_2&1\\[4pt]
-1&0&0&0&1&a_3\\[4pt]
-a_1&-1&0&0&0&1\\[4pt]
-1&-a_2&-1&0&0&0\\[4pt]
0&-1&-a_3&-1&0&0
\end{array}
\right)
=a_1a_2a_3-a_1-a_2-a_3.
$$

\medbreak \noindent {\sc Remark.}
Surprisingly, symmetric matrices of the same form
are also related to the rotundus:
$$
\det\left(
\begin{array}{cccccccccc}
&&&\;\;1&a_1&1&\\[2pt]
&&&&1&a_2&1&\\
&&&&&\ddots&\ddots&\ddots\\
&&&&&&\ddots&\ddots&1\\
1&&&&&&&1&a_{n}\\[2pt]
a_1&1&&&&&&&1\\
1&\ddots&\!\!\!\!\!\!\ddots&&&&&&\\
&\ddots&\ddots&\\
&&&\;\;1&&&&\\[2pt]
&&1&a_{n}&\!\!1&&
\end{array}
\right) = (-1)^n\bigl(R_n(a_1,\ldots,a_n)^2-4\bigr).
$$

\medbreak \noindent
{\it Proof of Theorem~\ref{PfLem}.\/}
Regard the matrix in~(\ref{TheOmEq}) as a
$2 \times 2$ block matrix with $n \times n$ entries.
As such, it has the form
$$
\left( \begin{matrix} \ E & C\\ -C & E \end{matrix} \right),
$$
where $C$ is the tridiagonal continuant matrix in~(\ref{ContEq}), and
$E$ is the skew-symmetric matrix with
a~$1$ in the upper right corner, a~$-1$ in the lower left corner,
and all other entries zero.

It clarifies the situation to prove a more general result.
Given any $n \times n$ matrix $A$, let us write $A_\midmatrix$
for the $(n-2) \times (n-2)$ matrix obtained from $A$
by removing its ``perimeter'': its first and last rows and columns.
We will prove that for any scalars $x$ and $y$,
\begin{equation} \label{PrfEq}
\det \left( \begin{matrix} xE & A\\ -A & yE \end{matrix} \right)
= \bigl( \det(A) - xy \det(A_\midmatrix) \bigr)^2.
\end{equation}
Taking $x$ and $y$ to be~$1$ and $A$ to be $C$ then gives the theorem.

Write $B$ for the matrix in~(\ref{PrfEq}).
Clearly $\det(B)$ is quadratic in both $x$ and $y$,
and it is a perfect square because $B$ is skew-symmetric.
Consequently it must take the form
$$
\det(B) = \bigl( \Delta_0 + x \Delta_x + y \Delta_y + xy \Delta_{xy} \bigr)^2
$$
for some polynomials $\Delta_0$, $\Delta_x$, $\Delta_y$, and $\Delta_{xy}$
in the entries of $A$, which are determined up to a single overall choice of sign.
Observe that
$$
\det(B) = \det \left[
\left( \begin{matrix} 0 & -\Id \\ \Id & \ 0 \end{matrix} \right)
\left( \begin{matrix} xE & A \\ -A & yE \end{matrix} \right)
\right] =
\det \left( \begin{matrix} A & -yE\\ xE & \ A \end{matrix} \right).
$$
Therefore if either $x$ or $y$ is zero, $\det(B) = \det(A)^2$.
Hence $\Delta_x = \Delta_y = 0$,
and we may take $\Delta_0 = \det(A)$.

Now use the following schematic diagram of $B$ to envision
the coefficient of $x^2 y^2$ in its determinant:
$$ B = 
\left( \begin{array}{ccc;{2pt/2pt}ccc}
&&\ \ x&&&\\
&&&&A&\\
-x&&&&&\\
\hdashline[2pt/2pt]
&&&&&\ \ y\\
&-A&&&&\\
&&&-y&&
\end{array} \right).
$$
It becomes clear that this coefficient is $\det(A_\midmatrix)^2$,
and so $\Delta_{xy}$ must be one of $\pm \det(A_\midmatrix)$.
The sign is negative, because $B$ is singular when
$x = y =1$ and $B = \Id$:
its first and last columns sum to~$0$.
\hfill $\Box$

\medbreak \noindent {\sc Comment.}
Theorem~\ref{PfLem} arises naturally in symplectic geometry.
Consider a ``projective $2n$-gon'' in $(2n-2)$-dimensional symplectic space,
i.e., a cyclically ordered configuration of $2n$ lines, satisfying the strong
``Lagrangian condition'' that every set of $n-1$ consecutive lines
generates a Lagrangian subspace.
It turns out that the moduli space of such configurations
is precisely the hypersurface where the rotundus vanishes.
The matrix in~(\ref{TheOmEq}) enters the picture as the Gram matrix
of the symplectic form evaluated on a certain normalized choice
of points on the lines of the configuration.

These geometric considerations are more technical and will be treated elsewhere.
In this note we restrict ourselves to combinatorial properties of the rotundus
which seem interesting and deserving of further study.

\section{Centrally symmetric triangulations}\label{TotSec}

Here we investigate the Diophantine equation
\begin{equation}
\label{DiffContEq}
R_n(a_1,\ldots,a_n)=0.
\end{equation}
We will show that it is an analog of
the Coxeter-Conway system~(\ref{CoCotEq}).
However, thanks to its cyclic invariance, one does not need a system:
a single equation contains complete information.

Let us first explain the classical Conway-Coxeter theorem~\cite{CoCo}.
An $n$-periodic solution $(a_i)_{i \in \Z}$ of the
system~(\ref{CoCotEq}) is called {\it totally positive\/} if 
\begin{equation}
\label{TotPostEq}
K_{j-i+1}(a_i,a_{i+1},\ldots,a_j)>0
\mbox{\rm\ for\ } j - i < n-3.
\end{equation}
Total positivity is one of the central notions of algebraic combinatorics.
The theorem is a beautiful combinatorial interpretation of
the totally positive solutions of~(\ref{CoCotEq}).
Given a triangulation of a (regular) $n$-gon,
let~$a_i$ be the number of triangles adjacent to the $i^\thup$ vertex.
This yields an $n$-periodic sequence of positive integers~$(a_i )_{i\in\Z}$.
The content of the theorem is that these sequences are
solutions of~(\ref{CoCotEq}), they are totally positive,
and every totally positive solution of~(\ref{CoCotEq})
arises in this way.

\begin{thmc} \cite{CoCo}
Totally positive integer solutions of~(\ref{CoCotEq}) correspond to triangulations of the $n$-gon.
\end{thmc}

For different proofs of this theorem, see~\cite{Hen,SVS}.

\medbreak \noindent {\sc Example.}
Up to cyclic permutation, the only totally
positive $5$-periodic integer solution of the system
$$
\left|
\begin{array}{cccc}
a_i&1&0\\[2pt]
1&a_{i+1}&\!\!\!1\\[2pt]
0&1&\!\!a_{i+2}
\end{array}\right|=1,
\quad 1\leq i\leq 5,
$$
is given by $(a_1, a_2,a_3, a_4,a_5)=(1, 3, 1, 2, 2)$. 
It corresponds to the only triangulation of the pentagon:
$$
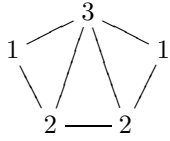
\xymatrix @!0 @R=0.50cm @C=0.5cm
{
&&3\ar@{-}[rrd]\ar@{-}[lld]\ar@{-}[lddd]\ar@{-}[rddd]&
\\
1\ar@{-}[rdd]&&&& 1\ar@{-}[ldd]\\
\\
&2\ar@{-}[rr]&& 2
}
$$
The label of each vertex is the number of triangles adjacent to it.

\medbreak
We now turn to the rotundus system~(\ref{DiffContEq}).
As usual, extend $(a_1, \ldots, a_n)$ to an $n$-periodic sequence $(a_i)_{i \in \Z}$.
By analogy with~(\ref{CoCotEq}), solutions of~(\ref{DiffContEq})
are said to be {\it totally positive\/}
if they satisfy~(\ref{TotPostEq}) for all $j-i \leq n$.
Such solutions are described by the following theorem.

\begin{thm}
\label{IntSolCor}
Every totally positive integer solution of~(\ref{DiffContEq}) corresponds to a
centrally symmetric triangulation of a $2n$-gon.
\end{thm}

\medbreak \noindent {\sc Example.}
Consider the following centrally symmetric triangulations of the decagon:
$$
 \xymatrix @!0 @R=0.32cm @C=0.45cm
 {
&&&5\ar@{-}[lllddd]\ar@{-}[lld]\ar@{-}[dddddddd]\ar@{-}[rrd]&
\\
&1\ar@{-}[ldd]&&&& 2\ar@{-}[rdd]\\
\\
2\ar@{-}[dd]&&&&&&2\ar@{-}[dd]\ar@{-}[lllddddd]\\
\\
2\ar@{-}[rrruuuuu]\ar@{-}[rdd]&&&&&&2\ar@{-}[ldd]\ar@{-}[lllddd]\\
\\
&2\ar@{-}[rrd]\ar@{-}[rruuuuuuu]&&&& 1\ar@{-}[lld]\\
&&&5\ar@{-}[rruuuuuuu]&
}
\qquad
\xymatrix @!0 @R=0.32cm @C=0.45cm
 {
&&&4\ar@{-}[lllddd]\ar@{-}[lld]\ar@{-}[dddddddd]\ar@{-}[rrd]&
\\
&1\ar@{-}[ldd]&&&& 3\ar@{-}[rdd]\ar@{-}[rdddd]\\
\\
3\ar@{-}[dd]\ar@{-}[rdddd]&&&&&&1\ar@{-}[dd]\\
\\
1\ar@{-}[rdd]&&&&&&3\ar@{-}[ldd]\ar@{-}[lllddd]\\
\\
&3\ar@{-}[rrd]\ar@{-}[rruuuuuuu]&&&& 1\ar@{-}[lld]\\
&&&4\ar@{-}[rruuuuuuu]&
}
\qquad
\xymatrix @!0 @R=0.32cm @C=0.45cm
 {
&&&4\ar@{-}[lllddd]\ar@{-}[lld]\ar@{-}[dddddddd]\ar@{-}[rrd]&
\\
&1\ar@{-}[ldd]&&&& 2\ar@{-}[rdd]\ar@{-}[rdddd]\\
\\
4\ar@{-}[dd]\ar@{-}[rdddd]&&&&&&1\ar@{-}[dd]\\
\\
1\ar@{-}[rdd]&&&&&&4\ar@{-}[ldd]\ar@{-}[lllddd]\ar@{-}[llluuuuu]\\
\\
&2\ar@{-}[rrd]&&&& 1\ar@{-}[lld]\\
&&&4\ar@{-}[llluuuuu]&
}
$$
\begin{center}
Totally positive solutions from triangulations.
\end{center}
\medbreak \noindent
At $n=5$, one easily checks that the values
$$
(5,2,2,2,1), \qquad \qquad (4,3,1,3,1), \qquad \qquad (4,2,1,4,1),
$$
of $(a_1,a_2,a_3,a_4,a_5)$ obtained from these triangulations
are indeed totally positive solutions of~(\ref{DiffContEq}).

\medbreak \noindent {\it Proof of Theorem~\ref{IntSolCor}.\/}
We deduce the result directly from the Conway-Coxeter theorem.
Recall that~(\ref{DiffContEq}) is the {\it zero-trace} condition
for the matrix $M_n$ in~(\ref{SLBisEq}).
In light of the obvious fact that this matrix has determinant~$1$,
(\ref{DiffContEq}) is equivalent to the condition that
$M_n$ have eigenvalues $\pm i$,
or in other words, $M_n^2=-\Id$.

This implies that the ``double'' $2n$-tuple $(a_1,\ldots,a_n,a_1,\ldots,a_n)$
is a solution of the Conway-Coxeter system of order $2n-2$.
By the Conway-Coxeter theorem, this $2n$-tuple must be given by
a triangulation of a $2n$-gon.
This triangulation is clearly centrally symmetric.

To prove the converse, one needs the fact that~(\ref{CoCotEq}) implies
$$
K_{n-1}(a_i,\ldots,a_{i+n-2})=0, \qquad
K_{n}(a_i,\ldots,a_{i+n-1})=-1.
$$
Indeed, this holds because the matrices
$M_{n-1}$ and $M_n$ have determinant~$1$.
Given a centrally symmetric triangulation of a $2n$-gon,
i.e., a totally positive solution of the Conway-Coxeter system of order $2n-2$,
we have shown that $M_{2n}=M_n^2=-\Id$.
Hence the result.
\hfill $\Box$

\medbreak \noindent {\sc Remark.}
If the assumption of total positivity is dropped,
the classification of integer solutions of~(\ref{CoCotEq}) is unknown,
even if we restrict to the cases for which
the $a_i$ themselves are positive; see \cite{Cun}.
Similarly, the classification of positive integer solutions of~(\ref{DiffContEq})
with $n \ge 4$ is an open problem.
For $n=5$, the simplest positive
but not totally positive solution of~(\ref{DiffContEq})
is $(a_1,a_2,a_3,a_4,a_5)=(2,1,1,1,1)$.
It cannot be obtained from a triangulation of the $10$-gon.

\section{Chebyshev polynomials}\label{ChebSec}

The celebrated {\it Chebyshev polynomials} are sequences of
orthogonal polynomials in one variable
satisfying the recurrence
$$
P_{n+1}(x) = 2xP_n(x) - P_{n-1}(x).
$$
The two sets of ``initial conditions''
$P_0(x)=1,\;P_1(x)=x$ and $P_0(x)=1,\;P_1(x)=2x$
lead to two series of polynomials, called
the Chebyshev polynomials of the first and second kinds, respectively.
These two series are usually denoted by $T_n(x)$ and $U_n(x)$.
They start as follows:
$$
\begin{array}{ll}
T_0(x)=1, & U_0(x)=1,\\[4pt]
T_1(x)=x, & U_1(x)=2x,\\[4pt]
T_2(x)=2x^2-1, & U_2(x)=4x^2-1,\\[4pt]
T_3(x)=4x^3-3x, & U_3(x)=8x^3-4x,\\[4pt]
T_4(x)=8x^4-8x^2+1, & U_4(x)=16x^4-12x^2+1,\\[4pt]
\cdots&
\end{array}
$$

It is well known that substituting
$a_1=a_2=\cdots=a_n=2x$
into the continuant $K_n$ gives precisely the 
Chebyshev polynomials of the second kind:
$$
\textstyle
U_n\left(\frac{x}{2}\right)=K_n(x,\ldots,x).
$$
As may be seen for example in~\cite{Aig},
this determinantal expression is useful in combinatorics.
A similar expression for the Chebyshev polynomials of the first kind
appears to be missing.

Applying our results, we obtain the ``Pfaffian formula''
\begin{equation}
\label{PaffEq}
T_n\left(\frac{x}{2}\right)=
\frac{1}{2}\;R_n(x,\ldots,x)=
\frac{1}{2}\;\pf\left(
\begin{array}{cccccccccc}
&&&\;\;1&x&1&\\[2pt]
&&&&1&x&1&\\
&&&&&\ddots&\ddots&\ddots\\
&&&&&&\ddots&\ddots&1\\
-1&&&&&&&1&x\\[2pt]
-x&\!\!\!-1&&&&&&&1\\
-1&\ddots&\!\!\!\!\ddots&&&&&&\\[2pt]
&\ddots&\ddots&\\
&&&-1&&&&\\[2pt]
&&\;-1&-x&\!\!\!-1&&
\end{array}
\right),
\end{equation}
the matrix being of size $2n\times 2n$.
This is an immediate corollary of Theorem~\ref{PfLem}, together with the well-known
(and obvious) relation between the polynomials of first and second kind:
$$
\textstyle
T_n(x)=\frac{1}{2}\bigl(U_n(x)-U_{n-2}(x)\bigr).
$$
We did not find~(\ref{PaffEq}) in the literature.

Applying~(\ref{SLBisEq}) and~(\ref{RotDefEq}), we have also the ``trace formula''
\begin{equation}
\label{PafTrEq}
T_n\left(\frac{x}{2}\right)=
\frac{1}{2}\;\tr
\left(
\begin{array}{cc}
x&1\\[4pt]
-1&0
\end{array}
\right)
\left(
\begin{array}{cc}
x&1\\[4pt]
-1&0
\end{array}
\right)
\cdots
\left(
\begin{array}{cc}
x&1\\[4pt]
-1&0
\end{array}
\right).
\end{equation}

\bigbreak \noindent
{\bf Acknowledgements}.
We are grateful to Sophie Morier-Genoud, Sergei Tabachnikov, and Richard Schwartz
for enlightening discussions.
C.C.\ was partially supported by Simons Foundation Collaboration Grant~207736.

\end{document}